\documentclass[10pt, a4paper]{article}
\usepackage{amsmath, amsfonts}
\usepackage[usenames]{color}
\usepackage[latin1]{inputenc}
\usepackage{hyperref} 
\pdfoutput=1
\hypersetup{bookmarks = true, 
                  colorlinks = true, 
                  linkcolor = blue, 
                  citecolor = blue, 
                  }
\usepackage{graphicx}
\usepackage{amssymb}
\usepackage{amsmath}

\newtheorem{theorem}{Theorem}[section]

\newtheorem{conjecture}[theorem]{Conjecture}
\newtheorem{corollary}[theorem]{Corollary}

\newtheorem{definition}[theorem]{Definition}
\newtheorem{example}[theorem]{Example}

\newtheorem{lemma}[theorem]{Lemma}

\newtheorem{proposition}[theorem]{Proposition}
\newtheorem{remark}[theorem]{Remark}

\newenvironment{proof}[1][Proof]{\noindent\textbf{#1.} }{\ \rule{0.5em}{0.5em}}

\title{\textbf{Whitney equisingularity of families of surfaces in $\mathbb{C}^3$}}
\author{ \ \ \ \\{M.A.S. Ruas\footnote{M.A.S. Ruas: ICMC, Universidade de São Paulo, Av. do Trabalhador São Carlense,
Centro, São Carlos, SP 13560-970, Brazil. e-mail: maasruas@icmc.usp.br} $\ \ $ and $\ \ $ O.N. Silva\footnote{O.N. Silva: ICMC, Universidade de São Paulo, Av. do Trabalhador São Carlense,
Centro, São Carlos, SP 13560-970, Brazil. e-mail: otoniel@icmc.usp.br, work supported by CAPES.}}}
\date{}

\begin{document}

\maketitle

\begin{abstract}

In this work, we study families of singular surfaces in $\mathbb{C}^3$ parametrized by $\mathcal{A}$-finitely determined map germs. We consider the topological triviality and Whitney equisingularity of an unfolding $F$ of a finitely determined map germ $f:(\mathbb{C}^2,0)\rightarrow(\mathbb{C}^3,0)$. We investigate the following conjecture: topological triviality implies Whitney equisingularity of the unfolding $F$? We provide a complete answer to this conjecture, given counterexamples showing how the conjecture can be false.
\end{abstract}

\section{Introduction}\label{section1}


$ \ \ \ \ $ Let $f:(\mathbb{C}^{n},0)\rightarrow (\mathbb{C}^{p},0)$ be a finitely determined map germ. Given a $1$-parameter unfolding $F:(\mathbb{C}^{n} \times \mathbb{C},0)\rightarrow (\mathbb{C}^{p} \times \mathbb{C},0)$ of $f$ defined by $F(x,t)=(f_{t}(x),t)$, we assume it is origin preserving, that is, $f_{t}(0)=0$ for any $t$. Then, we have a 1-parameter family of finitely determined map germs $f_{t}:(\mathbb{C}^{n},0)\rightarrow(\mathbb{C}^{p},0)$. 

By Thom's second isotopy lemma for complex analytic maps (\cite{gibson2}, Theorem $5.2$), every unfolding $F$ of $f$ which is Whitney equisingular is also topologically trivial. However, we know that the converse is not true in general (for instance, see \cite{damon3}, Example $6.2$).

The work of Lê and Teissier at the beginning of the 1980s led to characterizations of Whitney equisingularity by the constancy of a finite sequence of polar multiplicities, (see \cite{Teissier1} and \cite{Teissier2}). Some time later, Gaffney \cite{gaffney} gave necessary and sufficient conditions to characterize the Whitney equisingularity of families of finitely determined map germs $f_{t}:(\mathbb{C}^n,0)\rightarrow (\mathbb{C}^p,0)$. For families $f_{t}:(\mathbb{C}^2,0)\rightarrow (\mathbb{C}^3,0)$, he showed the following result:

\begin{theorem}\rm(\cite{gaffney} Theorem 8.7 and Corollary 8.9)\label{13} \textit{If $f:(\mathbb{C}^2,0)\rightarrow(\mathbb{C}^3,0)$ is finitely determined and $F$ is a $1$-parameter unfolding of $f$, then :}

\begin{flushleft}
\textit{(a) $F$ is Whitney equisingular if and only if $\mu(D(f_{t}))$, $m_{1}(f_{t}(\mathbb{C}^2),0)$ and $m_{0}(f_{t}(D(f_{t}))$ are constant.}\\

\textit{(b) If $f$ has corank $1$,  $F$ is Whitney equisingular if and only if $\mu(D(f_{t}))$ and $m_{0}(f_{t}(D(f_{t}))$ are constant.}
\end{flushleft}
\textit{where $D(f_{t})$ is the double point curve of $f_{t}$, $m_{1}(f_{t}(\mathbb{C}^2),0)$ is the first polar multiplicity of $f_{t}(\mathbb{C}^2)$ and $m_{0}(f_{t}(D(f_{t})))$ is the Hilbert-Samuel multiplicity $e(m_{f_{t}(D(f_{t}))},\mathcal{O}_{f_{t}(D(f_{t}))})$ where $m_{f_{t}(D(f_{t}))}$ is the maximal ideal of the local ring $\mathcal{O}_{f_{t}(D(f_{t}))}$ of $f_{t}(D(f_{t}))$}.
\end{theorem}

We have already mentioned above that if $F$ is Whitney equisingular then $F$ is topologically trivial. The first author conjectured in 1994 (see \cite{ruas2}) that in the case of families of finitely determined map germs $f_{t}:(\mathbb{C}^2,0)\rightarrow (\mathbb{C}^3,0)$ the converse is true. Moreover, conjectured that the unique necessary (and sufficient) invariant to control Whitney equisingularity is the Milnor number $\mu(D(f_{t}))$ of the double point curve of $f_{t}$.

\begin{conjecture}\rm(\cite{ruas2}, Ruas's conjecture) \label{conjecture1} \textit{Let $f:(\mathbb{C}^2,0)\rightarrow (\mathbb{C}^3,0)$ be a finitely determined map germ and $F:(\mathbb{C}^2 \times \mathbb{C},0)\rightarrow (\mathbb{C}^3 \times \mathbb{C},0)$ a $1$-parameter unfolding of $f$. Suppose that $\mu(D(f_{t}))$ is constant, then $F$ is Whitney equisingular.}
\end{conjecture}

Unsuccessful attempts to prove this conjecture were presented in \cite{ruas} and \cite{ruas3}. This problem remained unsolved for almost $25$ years. In this work, we provide a complete answer to this question, given counterexamples which show different ways in which the conjecture may fail.

\section{Double point spaces}\label{section2}

$ \ \ \ \ $ Consider a finite and holomorphic map $f:U\rightarrow \mathbb{C}^3$, where $U$ is a small enough neighbourhood of $0$ in $\mathbb{C}^2$. Throughout the paper, $(x,y)$ and $(X,Y,Z)$ are used to denote systems of coordinates in $\mathbb{C}^2$ and $\mathbb{C}^3$, respectively. The double point curve of $f$, denoted by $D(f)$, is defined as the following set

\begin{center}
$D(f):=\lbrace \ (x,y) \in U \ : \ f^{-1}(f(x,y)) \neq \lbrace (x,y) \rbrace \ \rbrace \ \cup \ \Sigma(f)$,
\end{center}

\noindent where $\Sigma(f)$ is the singular set of $f$. If $f$ is finite and generically $1$ $-$ $1$, then $D(f)$ is a closed analytic set of dimension $1$. We also consider the lifting of the double point curve $D^2(f) \subset U \times U$ given by the pairs $((x,y),(x^{'},y^{'}))$ such that either $f(x,y)=f(x^{'},y^{'})$ with $(x,y) \neq (x^{'},y^{'})$ or $(x,y)=(x^{'},y^{'})$ and $(x,y) \in \Sigma(f)$ (see \cite{marar1} and \cite{guilhermo}).\\

We need to choose convenient analytic structures for the double point curve $D(f)$ and the lifting of the double point curve $D^2(f)$. To do this, we follow the construction of \cite{marar1} which is also valid for holomorphic maps from $\mathbb{C}^n$ to $\mathbb{C}^p$, with $n<p$.

Let us denote the diagonals in $\mathbb{C}^2 \times \mathbb{C}^2$ and $\mathbb{C}^3 \times \mathbb{C}^3$ by $\Delta_2$ and $\Delta_3$, respectively, and denote the sheaves of ideals defining them by $\mathcal{I}_2$ and $\mathcal{I}_3$, respectively. Locally, $\mathcal{I}_{2}=\langle x-x^{'},y-y^{'} \rangle$ and $\mathcal{I}_{3}=\langle X-X^{'},Y-Y^{'}, Z-Z^{'} \rangle$. Since the pull-back $(f \times f)^{\ast}\mathcal{I}_{3}$ is contained in $\mathcal{I}_{2}$ and $U$ is small enough, then there exist functions $\alpha_{ij}\in \mathcal{O}_{U^{2}}$ well defined in all $U \times U$, such that

\begin{center}
$f_{i}(x,y)-f_{i}(x^{'},y^{'})= \alpha_{i1}(x,y,x^{'},y^{'})(x-x^{'})+ \alpha_{i2}(x,y,x^{'},y^{'})(y-y^{'})$, for $i=1,2,3.$
\end{center}

If $f(x,y)=f(x^{'},y^{'})$ and $(x,y) \neq (x^{'},y^{'})$, then every $2 \times 2$ minor of the matrix $\alpha=(\alpha_{ij})$ must vanish at $(x,y,x^{'},y^{'})$. We denote by $\mathcal{R}(\alpha)$ the ideal in $\mathcal{O}_{\mathbb{C}^{4}}$ generated by the $2\times 2$ minors of $\alpha$. Then we define the \textit{lifting of the double point curve $D(f)$} (with an analytic structure) as the complex space

\begin{center}
$ D^{2}(f)=V((f\times f)^{\ast}\mathcal{I}_{3}+\mathcal{R}(\alpha))$.
\end{center}

And we call \textit{double point ideal} the ideal $\mathcal{I}^{2}(f)=\langle(f\times f)^{\ast}\mathcal{I}_{3}+\mathcal{R}(\alpha)\rangle$. Although the ideal $\mathcal{R}(\alpha)$ depends on the choice of the coordinate functions of $f$, in \cite{mond1} it is proved that $\mathcal{I}^{2}(f)$ does not, and so $D^{2}(f)$ is well defined. It is easy to see that the points in the underlying set of $D^{2}(f)$ are exactly the ones in $U \times U$ of type $(x,y,x^{'},y^{'})$ with $(x,y) \neq (x^{'},y^{'})$, $f(x,y)=f(x^{'},y^{'})$ and of type $(x,y,x,y)$ such that $(x,y)$ is a singular point of $f$.

Let $f:(\mathbb{C}^2,0)\rightarrow(\mathbb{C}^{3},0)$ be a finite map germ and take a representative of $f$ defined on a small enough open neighborhood of the origin. Denote by $I_{3}$ and $R(\alpha)$ the stalks at $0$ of $\mathcal{I}_{3}$ and $\mathcal{R}(\alpha)$. We define the \textit{lifting of the double point space of the map germ $f$} as the complex space germ $D^{2}(f)=V((f \times f)^{\ast}I_{3}+R(\alpha))$.

\begin{remark}
Following Mond and Pellikan \rm\cite{mond2}, \textit{given a finite morphism of complex spaces $f:X\rightarrow Y$ then the push-forward $f_{\ast}\mathcal{O}_{X}$ is a coherent sheaf of $\mathcal{O}_{Y}-$modules and $\mathcal{F}_{k}(f_{\ast}\mathcal{O}_{X})$ denotes the \textit{k}th Fitting ideal sheaf. Notice that the support of $\mathcal{F}_{0}(f_{\ast}\mathcal{O}_{X})$ is just the image $f(X)$. Analogously, if $f:(X,x)\rightarrow(Y,y)$ is a finite map germ then we denote by $F_{k}(f_{\ast}\mathcal{O}_{X})$ the \textit{k}th Fitting ideal of $\mathcal{O}_{X,x}$ as $\mathcal{O}_{Y,y}-$module.}
\end{remark}

We now describe an appropriate analytic structure to $D(f)$ and one more space that is important to study the topology of $f(\mathbb{C}^{2})$, namely, the double point curve in the target, denoted by $f(D(f))$. 

\begin{definition}
\noindent (a) Consider $f:U\rightarrow \mathbb{C}^3$ as above, and let $p:D^{2}(f) \subset U \times U \rightarrow U$ be the restriction to $D^{2}(f)$ of the projection of $U \times U$ on the first factor. The \textit{double point curve} (with an analytic structure) is the complex curve

\begin{center}
$D(f)=V(\mathcal{F}_{0}(p_{\ast}\mathcal{O}_{D^2(f)}))$.
\end{center}

\noindent Set theoretically we have the equality $D(f)=p(D^{2}(f))$.\\

\noindent (b) The \textit{double point space in the target} is the complex curve $f(D(f))=V(\mathcal{F}_{1}(f_{\ast}\mathcal{O}_2))$. Notice that, the underlying set germ of $f(D(f))$ is the image of $D(f)$ by $f$.\\ 

\noindent (c) Given a finite map germ $f:(\mathbb{C}^{2},0)\rightarrow (\mathbb{C}^3,0)$, \textit{the germ of the double point curve} (with an analytic structure) is the germ of complex curve $D(f)=V(F_{0}(p_{\ast}\mathcal{O}_{D^2(f)}))$. \textit{The germ of the double point curve in the image} is the germ of the complex curve $f(D(f))=V(F_{1}(f_{\ast}\mathcal{O}_2))$, \rm(\textit{for details, see} \rm\cite{guilhermo} \textit{and} \rm\cite{mond2}).

\end{definition}

In \cite{marar1}, Marar and Mond showed that if $f:(\mathbb{C}^2,0)\rightarrow (\mathbb{C}^3,0)$ has corank $1$, then $f$ is finitely determined if and only if $\mu(D(f))$ is finite. The Milnor number of $D(f)$ is called \textit{Mond number}. In \cite{guilhermo}, Marar, Nuño-Ballesteros and Peñafort-Sanchis extended this criterion of finite determinacy for the corank $2$ case.

\begin{theorem}\rm(\cite{guilhermo} Theorem $2.4$ and Corollary $3.5$)\label{criterio} \textit{
Let $f:(\mathbb{C}^2,0)\rightarrow(\mathbb{C}^{3},0)$ be a finite and generically $1$ $-$ $1$ map germ. Then $D(f)$ is reduced if and only if $D^{2}(f)$ is reduced and the projection $p: D^{2}(f)\rightarrow (\mathbb{C}^{2},0)$ is generically $1$ $-$ $1$. Hence, $f$ is finitely determined if and only if its Mond number $\mu(D(f))$ is finite.}
\end{theorem}

\begin{definition} Let $f:(\mathbb{C}^2,0)\rightarrow (\mathbb{C}^3,0)$ be a finitely determined map germ. Take a representative $f:U\rightarrow V$ of $f$, where $U$ and $V$ are neighbourhoods of $0$ in $\mathbb{C}^2$ and $\mathbb{C}^3$ and consider an irreducible component $D(f)^j$ of $D(f)$.\\

\noindent (a) If the restriction ${f_|}_{D(f)^j}:D(f)^j\rightarrow \mathbb{C}^3$ is generically $1$ $-$ $1$, we say that $D(f)^j$ is an \textit{identification component of} $D(f)$. In this case, there exists an irreducible component $D(f)^i$ of $D(f)$, with $i \neq j$, such that $f(D(f)^j)=f(D(f)^i)$. We say that $D(f)^i$ is the \textit{correspondent identification component to} $D(f)^j$ or that the pair $(D(f)^j, D(f)^i)$ is a \textit{pair of identification components of} $D(f)$.  \\

\noindent (b) If the restriction ${f_|}_{D(f)^j}:D(f)^j\rightarrow \mathbb{C}^3$ is generically $2$ $-$ $1$, we say that $D(f)^j$ is a \textit{fold component of} $D(f)$.
\end{definition}

\begin{example}
\textit{Let $f(x,y)=(x,y^2,xy^3-x^3y)$ be the singularity $C_3$ of Mond's list} \rm(\cite{mond1}). \textit{In this case, $D(f)=V(xy^2-x^3)$. Then $D(f)$ has three irreducible components given by}

\begin{center}
$D(f)^1=V(x), \ \ \ $ $D(f)^2=V(x+y) \ \ $ and $ \ \ D(f)^3=V(x-y)$. 
\end{center}

\textit{Notice that $D(f)^1$ is a fold component, while $(D(f)^2$, $D(f)^3)$ is a pair of identification components. Also, we have that $f(D(f)^1)=V(Z,X)$ and $f(D(f)^2)=f(D(f)^3)=V(Z,Y-X^2)$ (see Figure \rm\ref{figura38}).}\\

\begin{figure}[h]
\centering
\includegraphics[scale=0.3]{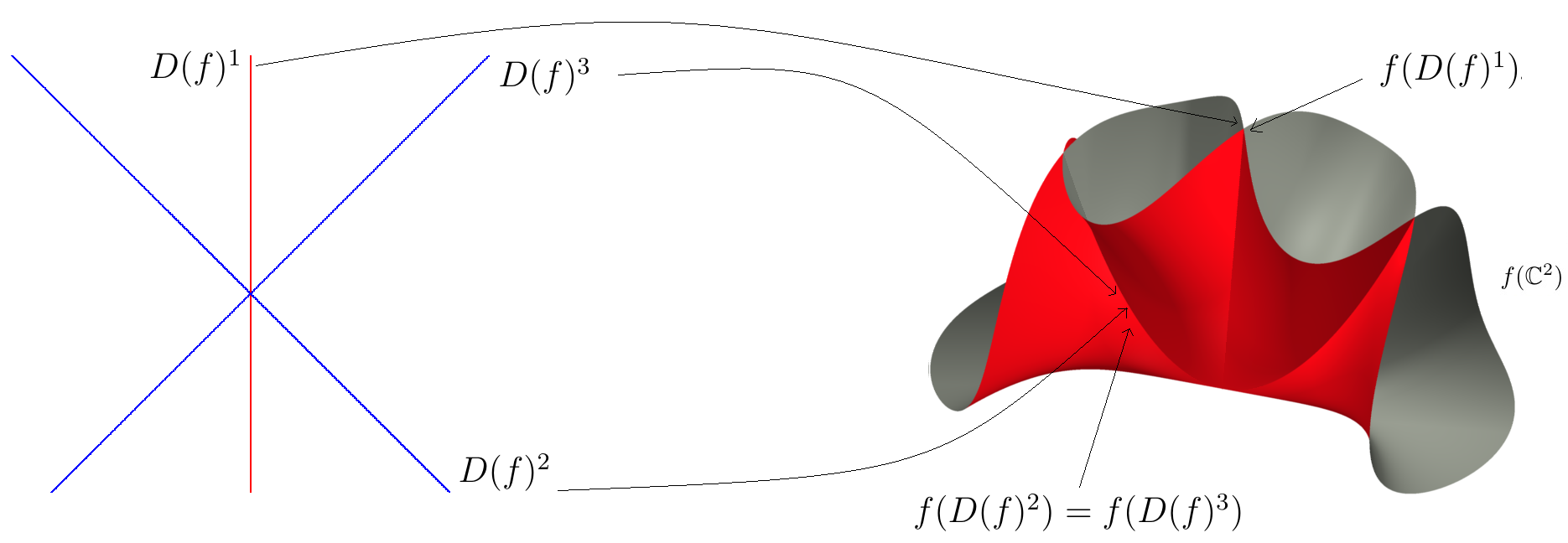}  
\caption{Identification and fold components of $D(f)$ (real points)}\label{figura38}
\end{figure}

\textit{Let $f:(\mathbb{C}^2,0)\rightarrow(\mathbb{C}^3,0)$ be a corank $1$ homogeneous map germ and write $f$ in the form $f(x,y)=(x,p(x,y),q(x,y))$. We will see in Section \ref{section6} that if the degrees of $p$ and $q$ are even, then $D(f)$ is defined by a homogeneous function of odd degree, then $D(f)$ has exactly one fold component.} 

\end{example}


We will see in the next section that the Milnor number of $D(f_{t})$ is an invariant that controls the topological triviality of the family of surfaces $f_{t}(\mathbb{C}^{2})$.

\section{Topological triviality of families of map germs}\label{section3}

$ \ \ \ \ $ Following \rm\cite{ruas}, we describe the necessary and sufficient conditions for a family $f_{t}$ to be topologically trivial. Firstly, we will define two more invariants. 

Let $f:(\mathbb{C}^2,0)\rightarrow (\mathbb{C}^3,0)$ be a finitely determined map germ and $F:(\mathbb{C}^{2} \times \mathbb{C},0)\rightarrow (\mathbb{C}^3 \times \mathbb{C},0)$ a $1$-parameter unfolding of $f$ defined by $F(x,t)=(f_{t}(x),t)$. We assume that the origin is preserved, that is, $f_{t}(0)=0$ for all $t$.

We say that a $1$-parameter unfolding $F$ is a stabilization of $f$ if there is a representative $F:U \times T\rightarrow \mathbb{C}^{3} \times T$, where $T$ and $U$ are open neighborhoods of $0$ in $\mathbb{C}$ and $\mathbb{C}^2$ respectively, such that $f_{t}:U\rightarrow \mathbb{C}^3$ is stable for all $t \in T \setminus \lbrace 0 \rbrace$.

By Mather-Gaffney's criterion \cite{mather}, there is a representative $f:U \subset \mathbb{C}^{2}\rightarrow V \subset \mathbb{C}^{3}$ such that $f^{-1}(0)=\lbrace 0 \rbrace$ and $f$ is stable on $U \setminus \lbrace 0 \rbrace$, where $V$ is an open neighborhoods of $0$ in $\mathbb{C}^3$. By shrinking $U$ if necessary, we can assume that there are no cross-caps or triple points in $U \setminus \lbrace 0 \rbrace$. Then, since we are in the nice dimensions, we can take a stabilization of $f$, $F:U \times D \rightarrow \mathbb{C}^4$, $F(z,s)=(f_{s}(z),s)$ where $D$ is a neighbourhood of $0$ in $\mathbb{C}$. We are ready to give the next definition.

\begin{definition}\label{def51}
We define $C(f) = \sharp  \lbrace$ cross-caps of $f_{s} \rbrace$ and $T(f)=\sharp \lbrace$ triple points of $f_{s} \rbrace$, for $s \neq 0$. These are analytic invariants of $f$ and they can be computed as follows \rm(\cite{mond1}):

\begin{center}
$C(f)=dim_{\mathbb{C}}\dfrac{\mathcal{O}_{2}}{Rf}$, $\ \ \ \ \ \ \ T(f)=dim_{\mathbb{C}}\dfrac{\mathcal{O}_{3}}{F_{2}(f_{\ast}\mathcal{O}_2)}$
\end{center}

\noindent where $Rf$ is the ideal generated by the maximal minors of the jacobian matrix of $f$.

\end{definition}

\begin{definition}
Let $F:(\mathbb{C}^2 \times \mathbb{C},0)\rightarrow (\mathbb{C}^3 \times \mathbb{C},0)$ be a $1$-parameter unfolding of a finitely determined map germ $f:(\mathbb{C}^2,0)\rightarrow(\mathbb{C}^3,0)$. We say that $F$ is \textit{$\mu$-constant} if $\mu(D(f_{t}))$ is constant for all $t \in T$, where $T$ is a neighborhood of $0$ in $\mathbb{C}$.
\end{definition}

\begin{definition} Let $F:(\mathbb{C}^2 \times \mathbb{C},0)\rightarrow (\mathbb{C}^3 \times \mathbb{C},0)$ be a $1$-parameter unfolding of a finitely determined map germ $f:(\mathbb{C}^2,0)\rightarrow(\mathbb{C}^3,0)$. We say that $F$ is \textit{topologically trivial} if there are germs of homeomorphisms:

\begin{center}
$\Phi :(\mathbb{C}^2 \times \mathbb{C},0)\rightarrow (\mathbb{C}^2 \times \mathbb{C},0), \ \ \ \ \ \ \ \  \ \ \ \Phi(x,t)=(\phi_{t}(x),t), \ \phi_{0}(x)=x, \ \phi_{t}(0)=0 $ 
\end{center}

\begin{center}
$\Psi:(\mathbb{C}^3 \times \mathbb{C},0)\rightarrow (\mathbb{C}^3 \times \mathbb{C},0) \ \ \ \ \ \ \ \  \ \ \ \Psi(x,t)=(\psi_{t}(x),t), \ \psi_{0}(x)=x, \ \psi_{t}(0)=0 $
\end{center}

\noindent such that $I=\Psi^{-1} \circ F \circ \Phi$, where $I(x,t)=(f(x),t)$ is the trivial unfolding of $f$.

\end{definition}

The following theorem characterizes topological triviality.

\begin{theorem}\rm{(\cite{ruas} Theorem $6.2$ and \cite{bobadilha} Corollary $32$)}\label{18}
\textit{Let $F:(\mathbb{C}^2 \times \mathbb{C},0)\rightarrow (\mathbb{C}^3 \times \mathbb{C},0)$ be a $1$-parameter unfolding of a finitely determined map germ $f:(\mathbb{C}^2,0)\rightarrow(\mathbb{C}^3,0)$. The following statements are equivalent:}

\begin{flushleft}
\textit{(1) $F$ is topologically trivial.}\\

\textit{(2) $F$ is $\mu$-constant.}
\end{flushleft}

\end{theorem}

The following theorem has been shown in \cite{marar1} in the case where $f$ has corank $1$ and in \cite{guilhermo} in the case where $f$ has corank $2$. This result is very useful for the computations of the numerical invariants.

\begin{theorem}\rm{(\cite{marar1} Theorem $3.4$ and \cite{guilhermo} Theorem $4.3$)}\label{40}
\textit{Let} $f:(\mathbb{C}^2,0)\rightarrow (\mathbb{C}^3,0)$ \textit{be finitely determined. Then}

\begin{flushleft}
(a) $\mu(D(f))=\mu(D^{2}(f))+6T(f)$\\
(b) $\mu(D^2(f))=2\mu(D^{2}(f)/S_{2})+C(f)-1$\\
(c) $\mu(D(f))=2\mu(f(D(f))+C(f)-2T(f)-1$.
\end{flushleft}

\noindent \textit{where $D^{2}(f)/S_{2}$ is the quotient complex space of $D^{2}(f)$ by the group $S_{2}=\mathbb{Z}/2$.}

\end{theorem}

The next corollary follows immediately by the upper semi-continuity of the invariants $D^{2}(f)$, $D^{2}(f)/S_{2}$, $C(f)$ and $T(f)$ and the formulas in Theorem \ref{40}.

\begin{corollary} Let $F:(\mathbb{C}^2 \times \mathbb{C},0)\rightarrow (\mathbb{C}^3 \times \mathbb{C},0)$ be a $1$-parameter unfolding of a finitely determined map germ $f:(\mathbb{C}^2,0)\rightarrow(\mathbb{C}^3,0)$. If $\mu(D(f_{t}))$ is constant, then $C(f_{t})$ and $T(f_{t})$ are constant.
\end{corollary} 
 
\begin{theorem}\label{mondformulas} (\rm{\cite{mond3}}) \textit{Let $f:(\mathbb{C}^2,0)\rightarrow (\mathbb{C}^3,0)$ be a quasi-homogeneous finitely determined map germ. Write $f(x,y)=(f_1,f_2,f_3)$ and denote by $d_{i}$ the degree of $f_i$. Let $\omega_{1},\omega_{2}$ be the weights of $x$ and $y$, respectively. If $\epsilon = d_{1}+d_{2}+d_{3}-\omega_{1}-w_{2}$ and $\delta=d_{1}d_{2}d_{3}/(\omega_{1}\omega_{2})$, then}

\begin{center}
$C(f)=\dfrac{1}{\omega_1\omega_2}((d_2-\omega_1)(d_3-\omega2)+(d_1-\omega_2)(d_3-\omega_2)+(d_1-\omega_1)(d_2-\omega_1))$,
\end{center}

\begin{center}
$T(f)=\dfrac{1}{6\omega_1\omega_2}(\delta-\epsilon)(\delta-2\epsilon)+\dfrac{C(f)}{3}, \ \ \ \ \ \ \ \ \ \ \mu(D(f))=\dfrac{1}{\omega_1\omega_2}(\delta-\epsilon-\omega_1)(\delta-\epsilon-\omega_2)$,
\end{center}


\end{theorem} 
 
\section{Whitney equisingularity of families of surfaces in $\mathbb{C}^3$}\label{section4}

$ \ \ \ \ $ Gaffney defined in \cite{gaffney} the excellent unfoldings. An excellent unfolding has a natural stratification whose strata in the complement of the parameter space $T$ are the stable types in source and target. For families $F$ as in previous section, the strata in the source are the following

\begin{center}
$\lbrace \mathbb{C}^2 \setminus D^2(F), \ D^2(F) \setminus T, \ T \rbrace$
\end{center}

In the target, the strata are:

\begin{center}
$\lbrace \mathbb{C}^3 \setminus F(\mathbb{C}^2 \times \mathbb{C}), \  F(\mathbb{C}^2 \times \mathbb{C}) \setminus \overline{F(D^{2}(F))}, \  F(D^2(F)) \setminus T, \ T \rbrace$.
\end{center}

Notice that $F$ preserves the stratification, that is, $F$ sends a stratum into a stratum.

\begin{definition}
An unfolding $F$ as above is \textit{Whitney equisingular} if the above stratifications in source and target are Whitney equisingular along $T$.
\end{definition}

In addition to Theorem \ref{13}, the following results are known:

\begin{theorem}\rm(\cite{ruas} Theorem $8.7$ and \cite{gaffney} Theorem $5.2$) \textit{Let $F$ be a $1$-parameter unfolding of a finitely determined map germ $f:(\mathbb{C}^2,0)\rightarrow (\mathbb{C}^3,0)$. If $F$ is $\mu$-constant, then $F$ is an excellent unfolding.}
\end{theorem}

In \cite{marar2}, Marar and Nuño-Ballesteros show that in the corank $1$ case the Whitney equisingularity of $f_{t}$ is characterized by the constancy of three invariants $C(f_{t})$, $J(f_{t})$ and $T(f_{t})$, where $J(f_{t})$ is the number of tacnodes that appear in a generic perturbation of the transversal section of $f_{t}$.

\begin{theorem}\rm{(\cite{marar2} Corollary $4.7$)}\label{16}
\textit{Let $F$ be a $1$-parameter unfolding of a finitely determined map germ $f:(\mathbb{C}^2,0)\rightarrow (\mathbb{C}^3,0)$ with corank $1$. Then, the following conditions are equivalent:}

\begin{flushleft}
\textit{(a) $F$ is Whitney equisingular.}
\end{flushleft}

\begin{flushleft}
\textit{(b) $C(f_{t})$, $J(f_{t})$ and $T(f_{t})$ are constant.}
\end{flushleft}

\end{theorem}

\begin{remark}\label{31}\rm Also in \cite{marar2}, they show that if $f:(\mathbb{C}^{2},0)\rightarrow(\mathbb{C}^3,0)$ is finitely determined, then $C(f)+2J(f)+6T(f)=\mu(D(f))+2m_{0}(f(D(f)))-1$. Therefore, the constancy of the invariants on the left side of the equality occurs if and only if $\mu(D(f))$ and $m_{0}(f(D(f)))$ are constant.
\end{remark}

Another useful result is the following lemma:

\begin{lemma}\rm(\cite{gaffney}, Proposition 8.4 and \cite{guilhermo}, Lemma 5.2)\label{5}
\textit{Let $f:(\mathbb{C}^2,0)\rightarrow(\mathbb{C}^3,0)$ be a finitely determined map germ. If $H \subset \mathbb{C}^3$ is a generic plane, $Y_{0}=H\cap f(\mathbb{C}^2)$ and $\tilde{Y}_{0}$ is the plane curve in $(\mathbb{C}^2,0)$ given by $\tilde{Y}_{0}=f^{-1}(H)$, we have:}

\begin{flushleft}
\textit{(a) $m_{1}(f(\mathbb{C}^2),0)=\mu(\tilde{Y}_{0},0)+ m_{0}(f(\mathbb{C}^2),0)-1$.}
\end{flushleft}

\begin{flushleft}
\textit{(b) $\mu_{1}(f(\mathbb{C}^2),0)=\mu(\tilde{Y}_{0},0)+2\cdot m_{0}(f(D(f)),0)$.}
\end{flushleft}

\end{lemma}

Item $(a)$ was shown in \cite{gaffney}. Item $(b)$ appears in \cite{guilhermo}, in which it is shown that the Whitney equisingularity of $f_{t}$ is characterized by the constancy of only two invariants, $\mu(D(f_{t}))$ and $\mu_{1}(f_{t}(\mathbb{C}^2),0)$.

\begin{theorem}\rm{(\cite{guilhermo} Theorem $5.3$)}\label{15}
\textit{Let $F$ be a $1$-parameter unfolding of a finitely determined map germ $f:(\mathbb{C}^2,0)\rightarrow (\mathbb{C}^3,0)$. Then, $F$ is Whitney equisingular if and only if $\mu(D(f_{t}))$ and $\mu_{1}(f_{t}(\mathbb{C}^2),0))$ are constant.}
\end{theorem}

The invariant $\mu_1(f(\mathbb{C}^2),0)$ is defined in \cite{guilhermo} as follows. As in Lemma \ref{5}, let $H$ be a generic plane in $\mathbb{C}^3$, $Y_0:=f(\mathbb{C}^2) \cap H$ and $\hat{Y}_0:=f^{-1}(H)$. For a generic $H$, $(Y_0,0)$ and $(\hat{Y}_0,0)$ are germs of reduced plane curves. Then, $\mu_1(f(\mathbb{C}^2),0):=\mu(Y_0,0)$. Let $F=(f_t(x),t)$ be an unfolding of $f$, for the next section, we will adopt the following notation, $(Y_{t},0):=(f_t(\mathbb{C}^2) \cap H,0)$ and $(\hat{Y}_{t},0):=(f_t^{-1}(H),0)$.

\section{Counterexamples to Ruas's conjecture}\label{section5}

$ \ \ \ \ $ In this section, we use Mond-Pellikaan's algorithm in \cite{mond1} to find a presentation matrix of a finite analytic map germ $g:(X,0)\rightarrow(\mathbb{C}^{n+1},0)$, where $(X,x)$ is a germ of Cohen-Macaulay analytic space of dimension $n$. For the computations we have made use of the software Singular \cite{singular} and the implementation of Mond-Pellikaan's algorithm given by Hernandes, Miranda, and Peñafort-Sanchis in \cite{hernandes}. At the webpage of Miranda \cite{aldicio} one can find a Singular library to compute presentation matrices based on the results of \cite{hernandes}.\\

The following example shows that in corank $2$ case the conjecture is false.

\begin{example}\label{oprimeiroexemplo}
Let $f:(\mathbb{C}^2,0)\rightarrow (\mathbb{C}^3,0)$ be defined by:

\begin{center}
$f(x,y)=(x^2, \ x^2y+xy^2+y^3, \ x^5+y^5)$
\end{center}

First we show that f is finitely determined. Denote by $(x,y,x^{'},y^{'})$ a point in $\mathbb{C}^2 \times \mathbb{C}^2$ and by $q^n$ the homogeneous function $q^n:\mathbb{C}^2\rightarrow \mathbb{C}$ defined by $q^n(w,w^{'})=w^n+w^{n-1}w^{'}+\cdots+ww^{'n-1}+w^{'n}$. Take a representative $f:U \rightarrow \mathbb{C}^3$ of the germ $f$ defined in a small enough neighborhood $U$ of $0$. Then the double point ideal $\mathcal{I}^2(f)$ is generated by $h_{1},\cdots, h_{6} \in \mathbb{C}\lbrace x,y,x^{'},y^{'} \rbrace$, where:

\begin{flushleft}
$h_{1}(x,y,x^{'},y^{'})=(x+x^{'})(x-x^{'}), \ \ \ \ \ \ \  \ \ \ \ \ \ \ \ \ \ \ \ \   h_{2}(x,y,x^{'},y^{'})= q^4(y,y^{'})(x+x^{'})$\\
\end{flushleft}

\begin{flushleft}
$h_{3}(x,y,x^{'},y^{'})=q^4(x,x^{'})(x-x^{'})+q^4(y,y^{'})(y-y^{'}),  \  h_{4}(x,y,x^{'},y^{'})=(x+x^{'}) \displaystyle\left(2q^2(y,y^{'})+(x+x^{'})(y+y^{'})+(x^2+x^{'2}) \right)$
\end{flushleft}

\begin{flushleft}
$h_{5}(x,y,x^{'},y^{'})= q^4(y,y^{'}) \displaystyle\left((y+y^{'})(x+x^{'})+(y^{2}+y^{'2})\right)- q^4(x,x^{'}) \displaystyle\left(2q^2(y,y^{'})+(x+x^{'})(y+y^{'})+(x^2+x^{'2})\right)$
\end{flushleft}

\begin{flushleft}
$h_{6}(x,y,x^{'},y^{'})=\displaystyle\left((y+y^{'})(x+x^{'})+(y^{2}+y^{'2})\right)(x-x^{'}) + \displaystyle\left(2q^2(y,y^{'})+(x+x^{'})(y+y^{'})+(x^2+x^{'2})  \right)(y-y^{'})$
\end{flushleft}

The projection $p: D^2(f) \subset U \times U \rightarrow U$, is just $(x, y, x^{'}, y^{'}) \mapsto (x,y)$. Then the ideal $\mathcal{F}_{0}(p_{\ast}\mathcal{O}_{D^{2}(f)})$ is generated by the determinant of the following presentation matrix $22 \times 22$ of $p_{\ast}\mathcal{O}_{D^{2}(f)}$, obtained by means of the Singular library mentioned above:

\fontsize{14}{14}
{\scriptsize $$\left( \begin{array}{cccccccccccccccccccccc}
 
y & -1 & 0 & 0 & 0 & 0 & 0 & 0 & 0 & 0 & 0 & 0 & 0 & 0 & 0 & 0 & 0 & 0 & 0 & 0 & 0 & 0\\

0 & y & -1 & 0 & 0 & 0 & 0 & 0 & 0 & 0 & 0 & 0 & 0 & 0 & 0 & 0 & 0 & 0 & 0 & 0 & 0 & 0\\

0 & x^2 & g_{14} & 0 & 0 & -x^2 & 0 & 0 & 0 & 0 & 0 & 0 & 0 & -1 & 0 & -1 & 0 & 0 & 0 & 0 & 0 & 0\\

0 & 0 & 0 & y & -1 & 0 & 0 & 0 & 0 & 0 & 0 & 0 & 0 & 0 & 0 & 0 & 0 & 0 & 0 & 0 & 0 & 0\\

x^3 & x^2 & x & x^2 & g_{14} & x^2 & x & 0 & x & 1 & x & 0 & 0 & 1 & 0 & 0 & 0 & 0 & 0 & 0 & 0 & 0\\

0 & 0 & 0 & 0 & 0 & y & -1 & 0 & 0 & 0 & 0 & 0 & 0 & 0 & 0 & 0 & 0 & 0 & 0 & 0 & 0 & 0\\

0 & 0 & 0 & 0 & 0 & 0 & y & -1 & 0 & 0 & 0 & 0 & 0 & 0 & 0 & 0 & 0 & 0 & 0 & 0 & 0 & 0\\

0 & 0 & 0 & 0 & 0 & 0 & x^2 & g_{14} & 0 & 0 & -x^2 & 0 & 0 & 0 & 0 & 0 & 0 & -1 & -1 & 0 & 0 & 0\\

0 & 0 & 0 & 0 & 0 & 0 & 0 & 0 & y & -1 & 0 & 0 & 0 & 0 & 0 & 0 & 0 & 0 & 0 & 0 & 0 & 0\\

0 & 0 & 0 & 0 & 0 & x^3 & x^2 & x & x^2 & g_{14} & x^2 & x & 0 & x & 1 & x & 0 & 1 & 0 & 0 & 0 & 0\\

0 & 0 & 0 & 0 & 0 & 0 & 0 & 0 & 0 & 0 & y & -1 & 0 & 0 & 0 & 0 & 0 & 0 & 0 & 0 & 0 & 0\\

0 & 0 & 0 & 0 & 0 & 0 & 0 & 0 & 0 & 0 & 0 & y & -1 & 0 & 0 & 0 & 0 & 0 & 0 & 0 & 0 & 0\\

0 & 0 & 0 & 0 & 0 & 0 & 0 & 0 & 0 & 0 & 0 & x^2 & g_{14} & 0 & 0 & -x^2 & 0 & 0 & 0 & 0 & -1 & -1\\

0 & 0 & 0 & 0 & 0 & 0 & 0 & 0 & 0 & 0 & 0 & 0 & 0 & y & -1 & 0 & 0 & 0 & 0 & 0 & 0 & 0\\

0 & 0 & g_{36} & 0 & g_{1} & 0 & g_{38} & g_{2} & 0 & xy & 0 & x^2 & x & x^2 & g_{15} & x^2 & x & 0 & -x & 0 & 0 & 0\\

0 & 0 & 0 & 0 & 0 & 0 & 0 & 0 & 0 & 0 & 0 & 0 & 0 & 0 & 0 & y & -1 & 0 & 0 & 0 & 0 & 0\\

0 & 0 & 0 & g_{42} & -x^2y & 0 & 0 & g_{39} & 0 & g_{3} & -x^3 & x^2 & -x & -x^2 & 0 & 0 & g_{16} & 0 & 0 & 0 & 0 & -1\\

0 & 0 & g_{37} & 0 & -g_{1} & 0 & -x^3 & -x^2 & 0 & -xy & 0 & 0 & 0 & 0 & -x & 0 & 0 & y & x & 0 & 1 & 0\\

0 & 0 & 0 & 0 & 0 & 0 & 0 & 0 & 0 & 0 & 0 & 0 & 0 & 0 & 0 & 0 & 0 & 0 & y & -1 & 0 & 0\\

0 & 0 & g_{5} & 0 & g_{6} & 0 & -x^4 & g_{7} & g_{42} & g_{8} & g_{21} & g_{23} & g_{24} & g_{25} & g_{26} & 0 & g_{27} & 0 & g_{28} & g_{17} & g_{29} & g_{20}\\

0 & 0 & g_{41} & 0 & g_{40} & 0 & -x^4 & g_{9} & 0 & -g_{1} & x^4 & -x^3 & g_{19} & 0 & 0 & 0 & 0 & 0 & x^2 & x & g_{14} & x\\

0 & 0 & g_{11} & 0 & g_{12} & 0 & x^4 & g_{13} & 2x^4 & g_{14} & g_{22} & g_{35} & g_{15} & g_{30} & g_{31} & 0 & g_{32} & 0 & g_{33} & g_{34} & g_{29} & g_{18}

\end{array} \right)$$}

\fontsize{10}{10}
\noindent where 
\begin{center}
$g_{1}=x^2y+xy^2$, $ \ g_{2}=2x^2+xy$, $ \ g_{3}=x^2+xy$, $ \ g_{4}=\dfrac{7}{10}x^3y+\dfrac{11}{10}xy^3$, $ \ g_{5}=\dfrac{-9}{10}x^2y^2-\dfrac{9}{5}xy^3$, $ \ g_{6}=\dfrac{-13}{10}x^3 - \dfrac{7}{10}x^2y+\dfrac{2}{5}xy^2$, $ \ g_{7}=\dfrac{-11}{10}x^2y - \dfrac{3}{2}xy^2$, $ \ g_{8}=-2x^3-2x^2y$, $ \ g_{9}=\dfrac{-3}{10}x^3y-\dfrac{9}{10}xy^3$, $ \ g_{10}=\dfrac{11}{10}x^2y^2+\dfrac{11}{10}xy^3$, $ \ g_{11}=\dfrac{7}{10}x^3- \dfrac{7}{10}x^2y - \dfrac{8}{5}xy^2$, $ \ g_{12}=\dfrac{9}{10}x^2y+\dfrac{3}{2}xy^2$, $ \ g_{13}=\dfrac{17}{5}x^2-xy$, $ \ g_{14}=x+y$, $ \ g_{15}=2x+y$, $ \ g_{16}=-x+y$, $ \ g_{17}=\dfrac{-7}{5}x+y$, $ \ g_{18}=\dfrac{17}{10}x+y$, $ \ g_{19}=-x^2-xy$, $ \ g_{20}=\dfrac{-13}{10}x$, $ \ g_{21}=\dfrac{11}{10}x^4$, $ \ g_{22}=\dfrac{-9}{10}x^4$, $ \ g_{23}=\dfrac{7}{5}x^3$, $ \ g_{24}=\dfrac{13}{5}x^2$,  $ \ g_{25}= \dfrac{19}{10}x^3$, $ \ g_{26}=\dfrac{6}{5}x^2$, $ \ g_{27}= \dfrac{-14}{5}xy$, $ \ g_{28}= \dfrac{13}{10}x^2$, $ \ g_{29}= \dfrac{x}{10}$, $ \ g_{30}= \dfrac{-21}{10}x^3$, $ \ g_{31}= \dfrac{4}{5}x^2$, $ \ g_{32}= \dfrac{11}{5}xy$, $ \ g_{33}= \dfrac{-17}{10}x^2$, $ \ g_{34}= \dfrac{8}{5}x$, $ \ g_{35}= \dfrac{8}{5}x^3$, $ \ g_{36}=-xy^2$, $ \ g_{37}=xy^2$, $ \ g_{38}=2x^3$, $ \ g_{39}=2x^2$, $ \ g_{40}=-x^2y^2$, $ \ g_{41}=-x^3y$ and $ \ g_{42}=-2x^4$.
\end{center}

\noindent Then,

\begin{center}
$D(f)=V(\mathcal{F}_{0}(p_{\ast}\mathcal{O}_{D^{2}(f)}))=V(3x^{22}+4x^{21}y+8x^{20}y^{2}+22x^{19}y^{3}+57x^{18}y^{4}+120x^{17}y^{5}+154x^{16}y^{6}+82x^{15}y^7-141x^{14}y^{8}-484x^{13}y^9-790x^{12}y^{10}-988x^{11}y^{11}-1064x^{10}y^{12}-1098x^9y^{13}-1102x^8y^{14}-1026x^7y^{15}-837x^6y^{16}-578x^5y^{17}-333x^4y^{18}-160x^3y^{19}-65x^2y^{20}-20xy^{21}-5y^{22})$.
\end{center}

Notice that $\mathcal{F}_{0}(p_{\ast}\mathcal{O}_{D^{2}(f)})$ is a homogeneous function of degree $22$. Moreover, it is not difficult to check with the help of a computer that it has isolated singularity. This implies that $f$ is finitely determined by Theorem \rm\ref{criterio}. \textit{Now, consider the following $1-$parameter unfolding $F=(f_{t}(x,y),t)$ of $f$ defined by:}

\begin{center}
$f_{t}(x,y)=(x^2+txy, \ x^2y+xy^2+y^3, \ x^5+y^5)$
\end{center}

\textit{We have that $f$ is homogeneous and that the unfolding $F$ only adds terms of same degree. This implies that
$F$ is topologically trivial by} \rm\cite{damon3}. \textit{Alternatively, a calculation shows that $\mu(D(f_t))=441$ for all $t$, then by Theorem} \rm \ref{18} \textit{$F$ is topologically trivial.} \textit{Choose constants $a, b$ and $c$ in $\mathbb{C}$ with $a\neq 0$ such that the plane $H$ defined by $H=V(aX+bY+cZ)$ is generic. In this way the family $(Y_{t},0)$ is given by}

\begin{center}
$(Y_{t},0)=V(a(x^2+txy)+b(x^2y+xy^2+y^3)+c(x^5+y^5))$
\end{center}

\textit{Then, $\mu(Y_{0},0)=2$ and $\mu(Y_{t},0)=1$ for $t \neq 0$. By Lemma} \rm\ref{5} \textit{and the upper semicontinuity of the invariants, $\mu_1(f_t(\mathbb{C}^2))$ can not be constant. Hence $F$ is not Whitney equisingular by Theorem} \rm\ref{15}.\\

 \textit{Using the formulas in Definition} \rm\ref{def51} \textit{and Theorem} \rm\ref{40}, \textit{we have that $C(f_t)=14$, $T(f_t)=56$ and $\mu(f_t(D(f_t)))=270$ for all $t$.}

\textit{The multiplicities $m_{0}(f_{t}(D(f_{t})))$ and $m_0(f_t(\mathbb{C}^2))$ remain constant. In fact, $m_{0}(f_{t}(D(f_{t})))=22$ and $m_0(f_t(\mathbb{C}^2))=6$ for all $t$. Moreover, by Lemma} \rm\ref{5} \textit{it follows that $\mu_{1}(f(\mathbb{C}^2),0)=46$, $\mu_{1}(f_{t}(\mathbb{C}^2),0)=45$ for $t\neq 0$, $m_{1}(f(\mathbb{C}^2),0)=7$ and $m_{1}(f_{t}(\mathbb{C}^2),0)=6$ for $t\neq 0$.}
\end{example}

\begin{remark} \textit{In} \rm\cite{guilhermo2}, \textit{Peñafort-Sanchis shows that if $n,m,k\geq 2$ are coprime integers, then the map germ}

\begin{center}
$f(x,y)=(x^n,y^m,(x+y)^k)$
\end{center}

\noindent \textit{is finitely determined. So a way to find a class of topologically trivial unfoldings with the property that $\mu(\hat{Y}_t,0)$ is not constant, as in the previous example, is the following:}\\

\textit{Let $n,m,k\geq 2$ be coprime integers, with $n<m<k$, and consider the map germ $f(x,y)=(x^n,y^m,(x+y)^k)$. Let be $F=(f_t(x,y),t)$ defined by}

\begin{center}
$f(x,y)=(x^n+ty^{n},y^m,(x+y)^k)$
\end{center}

\textit{The map germ $f$ is homogeneous and the unfolding $F$ only adds terms of same degree. Again, this implies that $F$ is topologically trivial by} \rm\cite{damon3}. \textit{Also,} 

\begin{center}
$\mu(\tilde{Y}_0,0)=(n-1)(m-1)>(n-1)(n-1)=\mu(\tilde{Y}_t,0)$, \textit{for} $t \neq 0$, 
\end{center}

\noindent \textit{then $F$ is not Whitney equisingular by Theorem} \rm\ref{15}.

\end{remark}

The following two examples show that if an unfolding $F$ of $f$ is topologically trivial, the multiplicity $m_0(f_t(D(f_t)))$ does not have to be constant. In the first $f$ has corank $2$, and in the second $f$ has corank $1$.

\begin{example}\label{osegundocontraex}
Consider the map germ $f:(\mathbb{C}^2,0)\rightarrow (\mathbb{C}^3,0)$ defined by

\begin{center}
$f(x,y)=(x^3,y^5,x^2-xy+y^2)$
\end{center}

It is not difficult to compute the double point curve $D(f)$ with the help of a computer and verify that $f$ is finitely determined, by just following the same routine calculations used in Example \rm\ref{oprimeiroexemplo}. \textit{In fact, $D(f)$ is a reduced curve given by}

$D(f)=V((81x^{16}-324x^{15}y+945x^{14}y^2-1971x^{13}y^3+3384x^{12}y^4-4716x^{11}y^5+5625x^{10}y^6-5664x^9y^7+5026x^8y^8-3900x^7y^9+2810x^6y^{10}-1840x^5y^{11}+1155x^4y^{12}-625x^3y^{13}+300x^2y^{14}-100xy^{15}+25y^{16})
(x^4-3x^3y+4x^2y^2-2xy^3+y^4)(x^2-xy+y^2))$

\textit{Now, consider the following unfolding $F=(f_t(x,y),t)$ of $f$}

\begin{center}
$f_t(x,y):=(x^3,y^5,x^2-xy+y^2+tx^2)$
\end{center}

\textit{Notice that $f$ is homogeneous and the unfolding $F$ only adds terms of same degree, then $F$ is topologically trivial. The presentation matrix of the push forward ${f_t}_{\ast}(\mathcal{O}_2)$ is given by}

\begin{center}
{$$ \footnotesize  \left( \begin{array}{ccccccccccccccc}

-Z &   0 &    1+t & 0 &    -1 &    0 &   1 &   0 &    0 &   0 &    0 &    0 &    0 &    0 &  0 \\
    
gX & -Z &    0&   0&    0&    -1&   0&    1&   0&   0&    0&    0&    0&    0&  0 \\
   
0&    gX& -Z &   X&    0&    0&   0&    0&    1&  0&    0&    0&    0&    0&  0\\
     
0&    0&    0&   -Z&        0&       1+t&0&    -1&    0&   1&   0&    0&    0&    0&  0\\
     
0&    0&    0&   gX&    -Z&       0&   0&    0&    -1&   0&    1&   0&    0&    0& 0\\ 
     
0&    0&    0&   0&       gX&    -Z&   X&    0&    0&   0&    0&    1&   0&    0& 0\\ 
     
0&    0&    0&   0&        0&       0&   -Z&    0&   1+t&   0&    -1&    0&    1&   0&  0\\
    
0&    0&    0&   0&        0&      0&   gX&    -Z&    0&   0&    0&    -1&    0&    1&  0\\ 
     
0&    0&    0&   0&        0&   0&   0&    gX&   -Z&   X&    0&    0&    0&    0&  1\\
    
Y&   0&    0&   0&        0&    0&   0&    0&    0&   -Z&    0&    1+t&    0&    -1&  0\\
     
0&    Y&   0&   0&        0&    0&   0&    0&    0&   gX&   -Z&    0&    0&    0& -1\\
     
0&    0&    Y&  0&        0&    0&   0&    0&    0&   0&    gX&   -Z&    X&    0&   0\\ 
     
0&    Y&    0&   Y&       0&    0&   0&    0&    0&   0&    0&    0&    -Z&    0&  1+t\\
  
0&    0&    Y&   0&       Y&   0&   0&    0&    0&   0&    0&    0&    gX&   -Z&  0\\

-XY&    0&    0&   0&       0&   Y&   0&    0&    0&   0&    0&    0&    0&   gX&  -Z\\

\end{array} \right)$$}
\end{center}

\noindent \textit{where $g=(1+t)$, then}  

\begin{center}
$f_t(\mathbb{C}^2)=V(F_0({f_t}_{\ast}(\mathcal{O}_2)))=V(Y^6+h_1XY^5Z+h_2X^5Y^3+h_3X^2Y^4Z^2+h_4X^6Y^2Z+h_5X^3Y^3Z^3-3Y^4Z^5+h_6X^{10}+h_7X^7YZ^2+h_8X^4Y^2Z^4+h_9XY^3Z^6+h_{10}X^8Z^3+h_{11}X^5YZ^5+h_{12}X^2Y^2Z^7+h_{13}X^6Z^6+h_{14}X^3YZ^8+3Y^2Z^{10}+h_{15}X^4Z^9+h_{16}XYZ^{11}+h_{17}X^2Z^{12}-Z^{15})$
\end{center}

\noindent \textit{where}

\begin{flushleft}
$h_1(x,y)=30t+15, \ \ $ $h_2(x,y)=15t^7-35t^6-147t^5-135t^4-20t^3+30t^2+15t+2, \  $ $h_3(x,y)=30t^3+315t^2+315t+90,$\\
$h_4(x,y)=15t^9-90t^8-735t^7-1785t^6-2070t^5-1125t^4-75t^3+225t^2+105t+15,$\\
$h_5(x,y)=375t^4+1600t^3+2145t^2+1215t+245,$\\
$h_6(x,y)=t^{15}+15t^{14}+105t^{13}+455t^{12}+1365t^{11}+3003t^{10}+5005t^9
+6435t^8+6435t^7+5005t^6+3003t^5+1365t^4+455t^3$
$+105t^2+15t+1,$\\
$h_7(x,y)=-90t^{10}-750t^9-2745t^8-5760t^7-7560t^6-6300t^5-3150t^4-720t^3+90t^2+90t+15,$\\
$h_8(x,y)=135t^6+1260t^5+3900t^4+5775t^3+4500t^2+1800t+300, \ \ \ \ \ $ $h_9(x,y)=15t-80,$\\
$h_{10}(x,y)=-5t^{12}-60t^{11}-330t^{10}-1100t^9-2475t^8-3960t^7-4620t^6-3960t^5-2475t^4-1100t^3-330t^2-60t-5,$\\
$h_{11}(x,y)=135t^7+1035t^6+3222t^5+5310t^4+4995t^3+2655t^2+720t+72, \ \ \ $ $h_{12}(x,y)=90t^3+45t^2-330t-345,$\\
$h_{13}(x,y)=10t^9+90t^8+360t^7+840t^6+1260t^5+1260t^4+840t^3+360t^2+90t+10,$\\
$h_{14}(x,y)=-225t^3-720t^2-765t-270, \ \ \ $ $h_{15}(x,y)=-10t^6-60t^5-150t^4-200t^3-150t^2-60t-10,$\\
$h_{16}(x,y)=-45t-60, \ \ \ \ \ \ \ \ $ $h_{17}(x,y)=5t^3+15t^2+15t+5$.
\end{flushleft}

\textit{Consider the plane $H$ defined by $X-Z=0$ which in this example is generic to $f$. Then, the family of plane curves $(f_t(\mathbb{C}^2)\cap H)=(Y_t,0)$ is defined in the coordinates $U$ and $W$ of $(\mathbb{C}^2,0)\simeq H$ by}

\begin{center}
$(Y_t,0)=V(W^6+h_1U^2W^5+h_2U^5W^3+h_3U^4W^4+h_4U^7W^2+h_5U^6W^3-3U^5W^4+h_6U^{10}+h_7U^9W+h_8U^8W^2+h_9U^7W^3+h_{10}U^{11}+h_{11}U^{10}W+h_{12}U^9W^2+h_{13}U^{12}+h_{14}U^{11}W+3U^{10}W^2+h_{15}U^{13}+h_{16}U^{12}W+h_{17}U^{14}-U^{15})$
\end{center}

\textit{Then, $\mu(Y_0,0)=47$ and $\mu(Y_t,0)=45$ for $t\neq 0$, hence $F$ is not Whitney equisingular by Theorem} \rm\ref{15}. \textit{Notice that $\mu(\tilde{Y}_t,0)=1$ for all $t$. Moreover, by Lemma} \rm\ref{5}, \textit{we have that $m_0(f(D(f)))=23$ and $m_0(f_t(D(f_t)))=22$ for $t \neq 0$.}

\textit{Notice also that the pair of irreducible components $D(f)^1 :=V(x-\alpha y)$ and $D(f)^2 :=V(x-\beta y)$, where $\alpha$ and $\beta$ are complex numbers such that $(x-\alpha y)(x-\beta y)=x^2-xy+y^2$, is a pair of identification components of $D(f)$.}

\textit{One of the main ideas of this example is that the map germ $f$ was constructed in such way that the image of the identification components $D(f)^1$ and $D(f)^2$ have multiplicity $3$, while all the other irreducible components of $f(D(f))$ have multiplicity $2$. When we take an appropriate deformation $f_t$ of $f$, all the irreducible components of $f_t(D(f_t))$ now have multiplicity $2$, for $t \neq 0$.} 
\end{example}

We now present the following counterexample to the conjecture in the corank $1$ case.

\begin{example}\label{oterceirocontraex}
\noindent Let $f$ be the map germ defined by $f(x,y)=(x,y^4,x^5y-5x^3y^3+4xy^5+y^6)$. As in the previous examples, it is not difficult to check, with the help of a computer, that $f$ is finitely determined. \textit{Now, consider the following unfolding $F=(f_t(x,y),t)$ of $f$}

\begin{center}
$f_t(x,y)=(x,y^4,x^5y-5x^3y^3+4xy^5+y^6+ty^7)$
\end{center}

\textit{The map germ $f$ is homogeneous and its unfolding $F$ only adds terms of same degree, then $F$ is topologically trivial}. \textit{The presentation matrix of the push forward ${f_t}_{\ast}(\mathcal{O}_2)$ is given by}

{$$  \left( \begin{array}{cccc}

-Z & X^5+4XY & Y & -5X^3+tY \\

-5X^3Y+tY^2 & -Z & X^5+4XY & Y \\

Y^2 & -5X^3Y+tY^2 & -Z & X^5+4XY \\

X^5Y+4XY^2 & Y^2 & -5X^3Y+tY^2 & -Z 

\end{array} \right)$$}

\textit{Then}

\begin{center}
$f_t(D(f_t))=V(YZ^2-Y^4+16X^2Y^2Z+8tXY^4+t^2Y^3Z-40X^4Y^3-10tX^3Y^2Z+33X^6YZ+2tX^5Y^3-10X^8Y^2+X^{10}Z,$\\
$Z^3-Y^3Z-16X^2Y^4-8tXY^3Z-t^2Y^5+40X^4Y^2Z+10tX^3Y^4-33X^6Y^3-2tX^5Y^2Z+10X^8YZ-X^{10}Y^2,$\\
$8XY^2Z+tYZ^2+tY^4-5X^3Z^2+59X^3Y^3-4t^2XY^4+2X^5YZ+40tX^4Y^3-52X^7Y^2-t2X^5Y^3+10tX^8Y^2-13X^{11}Y+X^{15},$\\
$8XY^4+2tY^3Z-74X^3Y^2Z-48tX^2Y^4-4t^2XY^3Z+X^5Z^2+241X^5Y^3+t^3Y^5+40tX^4Y^2Z-15t^2X^3Y^4-132X^7YZ+59tX^6Y^3-45X^9Y^2-4X^{11}Z-tX^{10}Y^2+5X^{13}Y)$
\end{center}

\textit{We can check that $m_0(f(D(f)))=9$ and $m_0(f_t(D(f_t)))=8$ for $t \neq 0$. Then $F$ is not Whitney equisingular by Theorem} \rm\ref{15}. \textit{Also, as $\mu(D(f_t))=196$, $C(f_t)=15$ and $T(f_t)=20$ for all $t$, by Remark} \rm\ref{31} \textit{it follows that $J(f)=39$ and $J(f_t)=38$ for $t\neq 0$. Then the constancy of the invariant $\mu(D(f_t))$ does not imply the constancy of the invariant $J(f_t)$.}

\textit{In this example, the curve $D(f)^1$ defined by $x=0$ is a fold component of $D(f)$. The image by $f$ of the component $D(f)^1$ has multiplicity $2$, while all the other irreducible components of $f(D(f))$ have multiplicity $1$. See Figure \ref{figure11} that shows five real irreducible components of $D(f)$ denoted by $D(f)^2=V(x-y), \ D(f)^3=V(x+y), D(f)^4=V(x+2y), \ D(f)^5=V(x-2y)$ (the others are not real components and do not appear in the figure). We create the appropriate deformation $f_t$ of $f$ so that all the irreducible components of $f_t(D(f_t))$ have multiplicity $1$, for $t \neq 0$.}

\begin{figure}[h]
\centering
\includegraphics[scale=0.35]{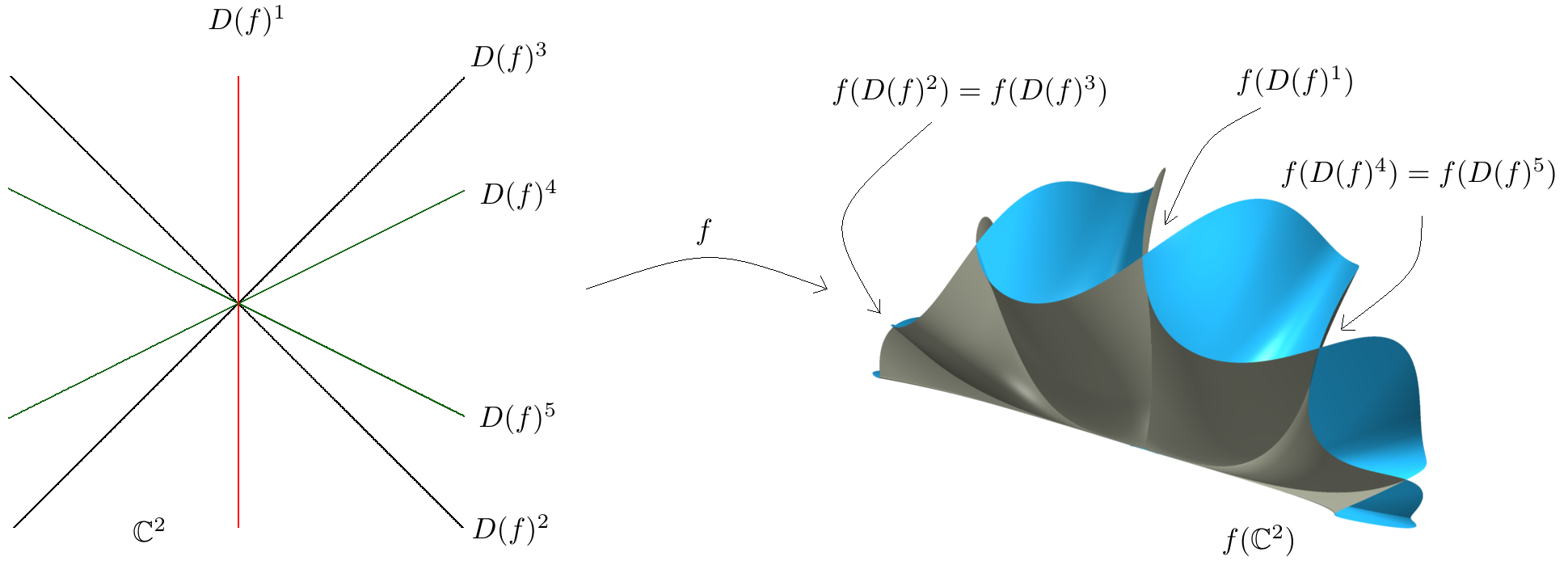}  
\caption{Five real components of $D(f)$ and the real image of $f(\mathbb{C}^2)$.}\label{figure11}
\end{figure}

\end{example}

We conclude this section with Table \ref{tabela1} which shows other counterexamples for the conjecture. The reader can check the details of each example following the same routine performed in the examples described above.

\begin{table}[!h]
  \centering 
  
  \begin{tabular}{lcccc}
\hline
$\begin{array}{c}\text{\footnotesize{Family}}\\ \text{\footnotesize{of map germs}} \end{array}$ & $\mu(\tilde{Y}_0,0)$  & $\mu(\tilde{Y}_t,0)$  & $m_0(f(D(f)))$  & $m_0(f_t(D(f_t)))$   \\
\hline \hline
 Corank $1$ &   &   &   &    \\
\hline  
& & & &\\
$(x,y^4,x^5y+xy^5+y^6+ty^7)$ &  0 & 0 & 9 & 8 \\
  & & & &\\ 
$(x,y^4,x^9y+xy^9+y^{10}+ty^{11})$ & 0 & 0 &  15 & 14\\
& & & &\\
$(x,y^4,x^{13}y+xy^{13}+y^{14}+ty^{15})$ & 0 & 0 & 21& 20\\
& & & &\\
$(x,y^4,x^{17}y+xy^{17}+y^{18}+ty^{19})$ & 0 & 0 & 27& 26\\
& & & &\\
$(x,y^6,x^{7}y+xy^{7}+y^{8}+ty^{9})$ & 0 & 0 & 20& 18\\
& & & &\\
$(x,y^6,x^{13}y+xy^{13}+y^{14}+ty^{15})$ & 0 & 0 & 35& 33\\
& & & &\\
\hline
 Corank $2$ &   &   &   &    \\
\hline
& & & &\\
$(x^2+txy,x^2y+xy^2+y^3,x^5+y^5)$ &  2 & 1 & 22 & 22 \\
  & & & &\\ 
$(x^3,y^5,x^2-xy+y^2+tx^2)$ & 1 & 1 &  23 & 22\\
& & & &\\
$\begin{array}{l}(x^n+ty^n,y^m,(x+y)^k) \\ _{(with \ \  2\leq n<m<k \ and \  n,m,k \  coprimes \  in \ pairs)} \\ _{ \ \ \ \ and  \ \ \ \ \ d=nmk-n-m-k+2} \end{array}$ & ${(n-1)(m-1)}$ & ${(n-1)^2}$ & ${{\dfrac{dn}{2}}}$  & ${\dfrac{dn}{2}}$ \\
& & & &\\
\end{tabular}
\caption{Counterexamples to the conjecture \ref{conjecture1}}\label{tabela1}
\end{table}

\section{Some formulas}\label{section6}

$ \ \ \ \ $ All examples of the previous section are deformations of homogeneous map germs. Inspired by them, we present in this section some formulas for the invariants $m_0(f(D(f)))$, $\mu_1(f(\mathbb{C}^2))$ and $J(f)$ in the case which $f$ is finitely determined, homogeneous and has corank $1$. We need the following lemma.

\begin{lemma}
Let $\alpha:\mathbb{C}\rightarrow \mathbb{C}^n$ be a map defined by 

\begin{center}
$\alpha(t)=(a_1t^{m_1},a_2t^{m_2}, \cdots, a_nt^{m_n})$, 
\end{center}

\noindent Then $\alpha$ is generically $d$-to-$1$ over the image $\alpha(\mathbb{C})$, where $d=gcd(m_1,m_2,\cdots,m_n)$, the greatest commom divisor of $m_1,m_2, \cdots, m_n$ .
\end{lemma}

\begin{proof}
It is not difficult to see that if $gcd(m_1,m_2,\cdots,m_n)=1$ then $\alpha$ is generically $1$-to-$1$ over its image. When $gcd(m_1,m_2,\cdots,m_n)\neq 1$, write $m_1=dq_1, \cdots, m_n=dq_n$ and consider the maps $\beta: \mathbb{C} \rightarrow \mathbb{C}$ and $\gamma: \mathbb{C} \rightarrow \mathbb{C}^n$ defined by

\begin{center}
$\beta(t)=t^d \ \ \ \ \ \ \ \ $ and $ \ \ \ \ \ \ \ \gamma(z)=(a_1z^{q_1}, \cdots, a_nz^{q_n})$.
\end{center}

\noindent Then $\alpha=\gamma \circ \beta$, $\beta$ is generically $d$-to-$1$, hence $\alpha$ is generically $d$-to-$1$.\end{proof}

\begin{proposition}\label{propdasformulas}
Let $f:(\mathbb{C}^2,0)\rightarrow (\mathbb{C}^3,0)$ be a homogeneous finitely determined map germ of corank $1$. Write $f$ in the form $f(x,y)=(x,p(x,y),q(x,y))$ and denote by $n$ and $m$ the degrees of $p$ and $q$, respectively, with $2\leq n \leq m$.\\

\noindent (a) $D(f)$ is the germ of a homogeneous curve with $d$ smooth irreducible components, where $d=nm-n-m+1$ and 

\begin{center}
$D(f)= V( \displaystyle {  \prod_{i=1}^{d}}(x-\alpha_iy))$
\end{center}

\noindent where $\alpha_i \in \mathbb{C}$.

\noindent (b) $d$ is odd if and only if $n$ and $m$ are both even.

\noindent (c) If $n$ and $m$ are both even, then $gcd(n,m)=2$, otherwise $f$ is not finitely determined. In this case, $D(f)$ has $d-1$ identification components, and $D(f)^j=V(x)$ is the unique fold component of $D(f)$ and the following hold:

\begin{center}
$m_0(f(D(f)))=\dfrac{nm-m}{2} \ \ \ \ \ \ \ \ \ \ \ \ \ $  $  \ \ \ \ \ \ \ \ \ \ \ \mu_1(f(\mathbb{C}^2))=nm-m$\\
$ \ \ $\\
$J(f)=\dfrac{m^2n+mn^2-m^2-7mn-n^2+6m+7n-6}{2}$
\end{center}

\noindent (d) If $n$ and $m$ are not both even, then each of the $d$ components of $D(f)$ is an identification component. In this case, we have that

\begin{center}
$m_0(f(D(f)))=\dfrac{nm-n-m+1}{2} \ \ \ \ \ \ \ \ \ $  $ \ \ \ \ \ \ \ \mu_1(f(\mathbb{C}^2))=nm-n-m+1$\\
$ \ \ $\\
$J(f)=\dfrac{m^2n+mn^2-m^2-7mn-n^2+6m+6n-5}{2}$
\end{center}

\end{proposition}

\begin{proof}
(a) The proof that $D(f)$ is homogeneous and $d=nm-n-m+1$ follows from (\cite{mond3}, Proposition $1.15$). We have that $D(f)=V(\lambda(x,y))$, where $\lambda(x,y)$ is a homogeneous polynomial of degree $d$. In this way, we can write

\begin{center}
$\lambda(x,y)= \displaystyle {  \prod_{i=1}^{d}}(\beta_i x-\gamma_i y)$, $ \ \ \ \ \ $ with $ \ \ \ \ \beta_i, \gamma_i \in \mathbb{C}$
\end{center} 

Let $f:U\rightarrow \mathbb{C}^3$ be a representative of $f$ and suppose that $D(f)^1=V(y)$ is an irreducible component of $D(f)$. Consider the map $n_1: U_1 \rightarrow U$, defined by $n_j(t)=(t,0)$, where $U_1$ is a small enough open neighborhood of $0$ in $\mathbb{C}$.  The map $f \circ n_1$ is generically $1$-to-$1$, then $D(f)^1$ is an identification component of $D(f)$. Any other component $D(f)^i$ of $D(f)$ is given by $D(f)^i=V(x-\alpha_iy)$, where $\alpha_i=\gamma_i/\beta_i$, then $f(D(f)^1)\cap f(D(f)^i)= \lbrace 0 \rbrace$, a contradiction. Then we can rewrite $\lambda$ in the form 

\begin{center}
$\lambda = \displaystyle {  \prod_{i=1}^{d}}(x-\alpha_i y)$, 
\end{center}

\noindent since $\beta_i \neq 0$, for all $i$.\\

(b) Write $n=2a$ and $m=2b$, then $nm-n-m+1=(2a)(2b)-(2a)-(2b)+1$ which is odd. Suppose now that $n=2a$ and $m=2b+1$, then $(2a)(2b+1)-(2a)-(2b+1)+1=4ab-2b$ which is even. Finally, suppose that $n=2a+1$ and $m=2b+1$, then $(2a+1)(2b+1)-(2a+1)-(2b+1)+1=4ab-2$ which is even.\\

(c) Consider the parametrization $n_j:U_j\rightarrow \mathbb{C}^2$ of each irreducible component $D(f)^j=V(x-\alpha_jy)$ of $D(f)$, defined by $n_j(t)=(\alpha_jt,t)$. If $\alpha_j \neq 0$, $f \circ n_j$ is generically $1$-to-$1$, then $D(f)^j$ is an identification component of $D(f)$. Since $n$ and $m$ are both even, $d$ is odd, then $D(f)$ has exactly
 $d-1$ identification components and $1$ fold component, denoted by $D(f)^1$, which is necessarily given by $D(f)^1=V(x)$. Write  $f$ in the form

\begin{center}
$f(x,y)=(x,a_0x^{n}+a_1x^{n-1}y+ \cdots + a_ny^n,b_0x^m+b_1x^{m-1}y+\cdots+ b_my^{m})$
\end{center}

Notice that $a_n \neq 0$ or $b_m \neq 0$, otherwise $f$ is not finite. We have three cases:

\begin{flushleft}
\textit{(c.1)} $f \circ n_1(t)=(a_nt^n,b_mt^m)$, if $a_n \neq 0$ and $b_m \neq 0$\\ 
\textit{(c.2)} $f \circ n_1(t)=(a_nt^n,0)$, if $a_n \neq 0$ and $b_m=0$\\
\textit{(c.3)} $f \circ n_1(t)=(0,b_mt^m)$, if $a_n=0$ and $b_m \neq 0$ 
\end{flushleft}

Let's look at each case.\\

\noindent \textit{Case (c.1)}: Let be $c=gcd(n,m)$, the map $f \circ n_1$ is generically $c-$to$-1$, hence $c=2$, otherwise $f$ is not finitely determined. Since $n\leq m$, it follows that $m_0(f(D(f)^1)))=\dfrac{n}{2}$.\\

\noindent \textit{Case (c.2)}: We have that $f \circ n_1(t)=(a_nt^n,0)$ is generically $2-$to$-1$, hence $n=2$. We can make coordinate changes at the source and target (see \cite{mond4}) such that $f$ can be written in the form 

\begin{center}
$f(x,y)=(x,y^2,yh(x,y^2))$
\end{center}

\noindent where $h(x,y^2)$ is a homogeneous polynomial of degree $m-1$. Since $2=n\leq m$, it follows that $m_0(f(D(f)^1)))=\dfrac{n}{2}$.\\

\noindent \textit{Case (c.3)}: This case is analogous to the previous case.\\

In any case we have that $gcd(n,m)=2$ and $m_0(f(D(f)^1)))=\dfrac{n}{2}$. Then

\begin{center}
$m_0(f(D(f)))=\displaystyle \left( \dfrac{d-1}{2} \right)+\dfrac{n}{2}=\dfrac{nm-n}{2}$
\end{center}

Using the formulas in Lemma \ref{5}, Remark \ref{31} and Theorem \ref{mondformulas}, we obtain the corresponding formulas for $\mu_1(f(\mathbb{C}^2))$ and $J(f)$.\\

(d) Consider again the parametrization $n_j:U_j\rightarrow \mathbb{C}^2$ of each component $D(f)^j=V(x-\alpha_jy)$ of $D(f)$, defined by $n_j(t)=(\alpha_jt,t)$. If $\alpha_j \neq 0$, $f \circ n_j$ is generically $1$-to-$1$, hence $D(f)^j$ is an identification component of $D(f)$. If there exists some fold component of $D(f)$, the unique possibility is when $\alpha_j=0$, then $D(f)$ can have only one fold component or each component of $D(f)$ is an identification component. As $n$ and $m$ are not both even, $d$ must be even. Since the number of identification components of $D(f)$ is always even, the number of fold components of $D(f)$ is even, so each component of $D(f)$ is an identification component. Since the image of each identification component is a smooth curve in $\mathbb{C}^3$, we have that

\begin{center}
$m_0(f(D(f)))=\dfrac{nm-n-m+1}{2}=\dfrac{d}{2}$
\end{center}

Again, using the formulas in Lemma \ref{5}, Remark \ref{31} and Theorem \ref{mondformulas}, we obtain the corresponding formulas for $\mu_1(f(\mathbb{C}^2))$ and $J(f)$.\end{proof}

\begin{proposition}
Let $f:(\mathbb{C}^2,0)\rightarrow (\mathbb{C}^3,0)$ be the map germ defined by $f(x,y)=(x^n,y^m,(x+y)^k)$, where $n,m$ and $k$ are coprime in pairs. If $n<m<k$, then $D(f)$ has $d=nmk-n-m-k+2$ smooth irreducible components and each component of $D(f)$ is an identification component. Furthermore 

\begin{center}
$m_0(f(D(f)))=\dfrac{dn}{2} \ \ \ \ \ \ \ \ \ \ $ and $ \ \ \ \ \ \ \ \ \ \ \mu_1(f(\mathbb{C}^2))=(n-1)(m-1)+dn$
\end{center}

\end{proposition}

\begin{proof}
By (\cite{mond3}, Proposition $1.15$) it follows that $D(f)$ is a homogeneous curve with $d=nmk-n-m-k+2$ smooth irreducible components. Since $n,m$ and $k$ are coprime in pairs and $n<m<k$, $d$ is always even. Let $D(f)^j=V(\beta_jx-\gamma_jy)$ be an irreducible component of $D(f)$ and consider its parametrization $n_j:U_j\rightarrow \mathbb{C}^2$. Again, since $n,m,k$ are coprime in pairs, the map $f \circ n_j$ is generically $1$-to-$1$, then $D(f)^j$ is an identification component of $D(f)$. Notice that $V(x)$ and $V(y)$ can not be irreducible components of $D(f)$. So the image of each curve $D(f)^j$ by $f$ is a curve in $\mathbb{C}^3$ with multiplicity $n$, and

\begin{center}
$m_0(f(D(f)))=\dfrac{dn}{2}$
\end{center}

\noindent Since $\mu(\tilde{Y}_0,0)=(n-1)(m-1)$, using Lemma \ref{5}, we obtain that $\mu_1(f(\mathbb{C}^2))=(n-1)(m-1)+dn$.\end{proof}

\section{A case in which the conjecture is true}\label{section7}

$ \ \ \ \ $  We now present a class of families of map germs in which the conjecture is true. Before, we will need the following lemma.

\begin{lemma}\label{lemmaauxiliarcap4}
Let $f:(\mathbb{C}^2,0)\rightarrow (\mathbb{C}^3,0)$ be a finitely determined map germ of corank $1$. Let $F=(f_t,t)$ be a topologically trivial $1$-parameter unfolding of $f$. If every irreducible component of $f(D(f))$ is smooth, then $F$ is Whitney equisingular.
\end{lemma}

\begin{proof}
Since the multiplicity $m_0(f(D(f)^j)))$ of each irreducible component of $f(D(f))$ is $1$, the result follows by the upper semi-continuity of the multiplicity.
\end{proof}

\begin{theorem}
Let $f:(\mathbb{C}^2,0)\rightarrow (\mathbb{C}^3,0)$ be a homogeneous finitely determined map germ of corank $1$. Write $f$ in the form $f(x,y)=(x,p(x,y),q(x,y))$ and let $n$ and $m$ be the degrees of $p$ and $q$, respectivelly. Let $F=(f_t,t)$ be a topologically trivial $1$-parameter unfolding of $f$. If $gcd(n,m)\neq 2$, then $F$ is Whitney equisingular.
\end{theorem}

\begin{proof}
By Proposition \ref{propdasformulas}, it follows that all components of $D(f)$ are identification components. Then the image of each component $D(f)^j$ by $f$ is a smooth curve in $\mathbb{C}^3$, and the result follows by Lemma \ref{lemmaauxiliarcap4}.
\end{proof}

\begin{flushleft}
\textit{Acknowlegments}. Thanks are due to the referee for many valuable comments and suggestions. We also would like to thank J.J. Nuño-Ballesteros for many helpful conversations. The first author acknowledges partial support by FAPESP, grant $2014/00304$-$2$ and CNPq, grant 306306/2015-8. The second author would like to thank CAPES for financial support.
\end{flushleft}


\small
\begin{thebibliography}{99}


\bibitem{ruas}R. Bedregal, K. Houston, M.A.S. Ruas, \textit{Topological triviality of families of singular surfaces}, preprint, arXiv:math/0611699v1 [math.CV], 2006.

\bibitem{bobadilha} J.F. de Bobadilla, M. Pe Pereira, \textit{Equisingularity at the Normalisation}. J. Topol. 1 (2008), no. 4, 879-909.




\bibitem{damon3} J. Damon, \textit{Topological triviality and versality for subgroups of A and K}. II. Sufficient conditions and applications. Nonlinearity 5 (1992), no.2, 373-412.


\bibitem{gaffney}T. Gaffney. \textit{Polar multiplicities and equisingularity of map germs}. Topology 32 no 1, (1993), 185-223.

\bibitem{gibson2} C. G. Gibson, K. Wirthmuller, A. A. du Plessis, E. J. N. Looijenga, \textit{Topological stability of smooth mappings}, Springer LNM 552, Springer-Verlag, 1976.

\bibitem{hernandes} M. E. Hernandes, A. J. Miranda, and G. Peñafort-Sanchis, \textit{A presentation matrix algorithm for $f_{\ast}\mathcal{O}_{X,x}$} Preprint, 2015.

\bibitem{marar1} W.L. Marar, D. Mond, \textit{Multiple point schemes for corank 1 maps}, J. London Math. Soc. (2) 39 (1989), no. 3, 553-567.


\bibitem{marar2}W. L. Marar, J. J. Nuño-Ballesteros, \textit{Slicing corank 1 map germs from $\mathbb{C}^2$ to $\mathbb{C}^3$}, Quart. J. Math. 65 (2014), 1375-1395; doi:10.1093/qmath/hat061

\bibitem{guilhermo}W. L. Marar, J. J. Nuño-Ballesteros, and G. Peñafort-Sanchis, \textit{Double point curves for corank 2 map germs from $\mathbb{C}^2$ to $\mathbb{C}^3$}, Topology Appl. 159 (2012), no. 2, 526-536.


\bibitem{aldicio} A. J. Miranda, https://sites.google.com/site/aldicio/publicacoes.

\bibitem{mond4}D. Mond, \textit{On the classification of germs of maps from $\mathbb{R}^2$ to $\mathbb{R}^3$}.

\bibitem{mond1}D. Mond, \textit{Some remarks on the geometry and classification of germs of map from surfaces to 3-space}. Topology 26(3), 361-383 (1987).

\bibitem{mond2}D. Mond, R. Pellikaan, \textit{Fitting ideals and multiple points of analytic mappings}. in: Algebraic Geometry and Complex Analysis, Patzcuaro, 1987, in: Lecture Notes in Math., vol 1414, Springer, Berlin, pp. 107-161, 1989.

\bibitem{mond3}D. Mond, \textit{The number of vanishing cycles for a quasihomogeneous mapping from $\mathbb{C}^2$ to $\mathbb{C}^3$}; Quart. J. Math. Oxford (2), 42 (1991), 335-345.



\bibitem{guilhermo2} G. Peñafort-Sanchis, \textit{Reflection Maps}, pre-print, arXiv:1609.03222v2 [math.AG], 2016.

\bibitem{ruas3}M.A.S. Ruas \textit{Equimultiplicity of topologically equisingular families of parametrized surfaces in $\mathbb{C}^{3}$}, preprint, arXiv:1302.5800v1 [math.CV], 2013.

\bibitem{ruas2} M.A.S. Ruas, \textit{On the equisingularity of families of corank 1 generic germs}, Differential Topology, Foliations and Group Action. Contemporary Mathematics, Vol.161, 113-121, (1994).


\bibitem{Teissier1}B. Teissier, \textit{Sur la classification des singularités des espaces analytiques complexes}, Proceedings of the International Congress of Mathematicians, Vol. 1, 2 (Warsaw, 1983), 763-781, PWN, Warsaw, 1984.

\bibitem{Teissier2}B. Teissier, \textit{Variétés polaires II: multiplicités polaires, sections planes, et conditions de Whitney}, in ``Algebraic Geometry, Proceedings, La Rabida, 1981," Lecture Notes in Math. 961, Springer-Verlag, 1982, pp. 314-491.

\bibitem{mather}C. T. C. Wall, \textit{Finite determinacy of smooth map-germs}, Bull. London Math. Soc. 13 (1981), no. 6, 481-539.

\bibitem{singular} D. Wolfram, G. M. Greuel, G. Pfister, and H. Schönemann, Singular 4-0-2, \textit{A computer algebra system for polynomial computations}, http://www.singular. uni-kl.de, 2015.



\end{thebibliography}
\end{document}